\documentclass[12pt,oneside]{amsart}

     \usepackage[pdftex]{graphicx}
     \usepackage{sidecap}

      \usepackage{amssymb}
      \usepackage{enumerate}	
      \usepackage{pgf}
      \usepackage{tikz}
      \usepackage{graphicx}
      \usetikzlibrary{arrows}
      \usetikzlibrary{shapes,snakes,calendar,matrix,backgrounds,folding}   
      \usepackage{listings}	

      \usepackage{caption}
      \usepackage{subcaption}

      \theoremstyle{plain}
      \newtheorem{theorem}{Theorem}[section]
      \newtheorem{lemma}[theorem]{Lemma}
      \newtheorem{corollary}[theorem]{Corollary}
      \newtheorem{proposition}[theorem]{Proposition}

      \theoremstyle{definition}
      \newtheorem{definition}[theorem]{Definition}
      
      \theoremstyle{rem}
      \newtheorem{rem}[theorem]{Remark}
      
      \theoremstyle{example}
      \newtheorem{example}[theorem]{Example}

\newcommand{\norm}[1]{\left\lVert#1\right\rVert}

      \newcommand{\R}{{\mathbb R}}

      \newcommand{\N}{\mathbb{N}}

      \newcommand{\NN}{\mathbb{N}}   
        \newcommand{\RR}{\mathbb{R}}   
        \newcommand{\ZZ}{\mathbb{Z}}

\usepackage{amsmath,amssymb,verbatim,graphicx,amsthm,color,url, enumerate}
\usepackage[top=1in, bottom=1in, left=1.25in, right=1.25in]{geometry}
\usepackage{cite}
\usepackage[backref=page,colorlinks=true]{hyperref} 
\usepackage{ifthen}
\usepackage{tikz}
\usetikzlibrary{decorations.pathreplacing}
\usepackage{tikz-cd}

\renewcommand*{\backref}[1]{}
\renewcommand*{\backrefalt}[4]{\tiny
	\ifcase #1 (\textbf{NOT CITED.})%
	\or    (Cited on page~#2.)%
	\else   (Cited on pages~#2.)%
	\fi}

      \makeatletter
      \def\@setcopyright{}
      \def\serieslogo@{}
      \makeatother
   
\begin{document}


   \title[$\Delta$-transitivity for several transformations]{$\Delta$-transitivity for several transformations and an application to the coboundary problem}


\author{Italo Cipriano}
\address{Facultad de Matem\'{a}ticas, Pontificia Universidad Cat\'olica de Chile, Santiago, Chile}
\email{icipriano@gmail.com}
\urladdr{http://www.icipriano.org/} 

\author{Ryo Moore}
\address{Facultad de Matem\'{a}ticas, Pontificia Universidad Cat\'olica de Chile, Santiago, Chile}
\email{rymoore@mat.puc.cl}
\urladdr{https://sites.google.com/view/ryomoore/} 

\begin{thanks}
{I.C. was partially supported by CONICYT PIA ACT172001.}
\end{thanks}

\begin{thanks}
	{R.M. was partially supported by  FONDECYT-CONICYT Postdoctorado 3170279.}
\end{thanks}

\begin{abstract}
Given a compact and complete metric space $X$ with several continuous transformations $T_1, T_2, \ldots T_H: X \to X,$ we find sufficient conditions for the existence of a point $x\in X$ such that  $(x,x,\ldots,x)\in X^H$ has dense orbit for the transformation $$\mathcal T:=T_1\times T_2\times\cdots\times T_H.$$

We use these conditions together with Liv\v{s}ic theorem, to obtain that for $\alpha$-H\"older maps $f_1,f_2,\ldots,f_H: X\to \R,$ the product $\prod_{i=1}^H f_i(x_i)$ is a smooth coboundary with respect to $\mathcal T$ is equivalent to the existence of a non-empty open subset $U \subset X$ such that
		$$\sup_{N} \sup_{x\in U}\left| \sum_{j=0}^{N} \prod_{i=1}^H f_i (T_i^{j} x) \right| < \infty.$$
\end{abstract}       
   \date{\today}
   \maketitle

\section{Introduction}
Let $(X, d)$ be a compact metric space, and $T_i: X \to X$ a continuous transformation (not necessarily all homeomorphisms) for $i = 1, 2, \ldots H$ for some integer $H \geq 2$. There are two items of interests in this note. Firstly, we will be interested in the properties of these maps for which the set
\[E_\Delta := \{x \in X: \text{ The set } \{(T_1^nx, T_2^nx, \ldots, T_H^nx): n \in \NN \} \text{ is dense in } X \} \]
is non-empty; in such case, we say the product map $T_1 \times T_2 \times \cdots \times T_H$ to be \textit{$\Delta$-transitive}. This topological multiple recurrence problem was originally studied by E. Glasner in \cite{Glasner_structure} when each transformation is a power of a single transformation (i.e. $T_i = T^i$, where $T$ is a continuous transformation on $X$). Later, T. K. S. Moothathu simplified the methods of Glasner \cite{Moothathu_diagonal}, and it was studied further by D. Kwietniak and P. Oprocha \cite{KwOp_mixing}. 

Generally, some variant of transitivity properties are assumed for each function, such as weak mixing, topological mixing, and/or minimality. It was shown by Moothathu that there exists a topologically mixing system for which the set $E_\Delta$ is empty \cite[Proposition 3]{Moothathu_diagonal}. Many of the previous studies on this topic has focused for the power of single transformations. In 2016, W. Huang, S. Shao, and X. Ye worked on this problem for several minimal and topologically weak mixing homeomorphisms that generate a nilpotent group \cite{HSY_topological}. We also note that I. Assani showed that if the maps $T_i$'s are commuting (measure theoretically) weakly mixing homeomorphisms that preserve a common probability measure $\mu$ for which the measure of every non-empty open set is positive, then $\mu(E_\Delta)$ equals to one; see the proof of \cite[Theorem 1]{Assani_v1}. One of the purpose of this note is to identify a sufficient condition for which the set $E_\Delta$ is nonempty without assuming every transformation is homeomorphism; in doing so, we consider additional topological conditions.

The second interest of this note is to obtain a Liv\v{s}ic type result for nonconventional ergodic sums (i.e. the sum that appears in the item (1) of Theorem \ref{AssaniThm}). In 2018, I. Assani announced a preprint with the following result:

\begin{theorem}[{\cite[Corollary 1.6 and the remark]{Assani_coboundary}}]\label{AssaniThm}
	Let $(X, \mu, T_1, T_2, \ldots, T_H)$ be a measure preserving system, where  $T_i$ is bi-measurable for each $1 \leq i \leq H$, and $f_1, f_2, \ldots, f_H \in L^\infty(\mu)$. The following statements are equivalent.
	\begin{enumerate}
		\item We have
		\[\sup_{N} \norm{ \sum_{n=1}^N \prod_{i=1}^H f_i \circ T_i^n }_{L^\infty(\mu)} < \infty. \]
		\item There exists a real-valued function $V$ on $X^H$ such that the map $x \mapsto V \circ \Phi^j(x, x, \ldots, x)$ is essentially bounded and $\mu$-measurable for every $j \in \ZZ$, and
		\[\prod_{i=1}^H f_i(x) = V(x, x, \ldots, x) - V \circ \Phi(x, x, \ldots, x) \text{ for $\mu$-a.e. } x \in X   \]
	\end{enumerate}
\end{theorem}

We note that Assani's result does not have any restrictions on each transformation aside from the fact that they preserve a common probability measure. We will show a topological analogue of such result, where we show the existence of continuous coboundary when a common invariant probability measure is absent. In particular, the $\Delta$-transitivity condition allows one to show that the boundedness of the nonconventional sum is equivalent of showing the product of the functions is a coboundary with respect to the product transformation.

\subsection*{Organization of the paper}
In section \ref{sec2} we discuss the conditions that guarantee $\Delta$-transitivity for the case we have two continuous transformations. The same is discussed in section \ref{sec3}, but for arbitrary many transformations. In section \ref{sec3}, we will discuss the coboundary problem.

\section{ $\Delta$-transitivity I}\label{sec1}

Let $X$ be a compact metric space and $T:X\to X,S:X\to X$ two continuous maps. We define the map $T\times S:X^2\to X^2,$ $(x,y)\mapsto (Tx,Sy).$ A natural problem is to characterise topologically $\Delta$-transitivity as property of $T$ and $S.$ The following theorem answer this.

\begin{theorem}\label{teo1}
Suppose that $T$ and $S$ are both weakly mixing and syndetically transitive, $T$ is homeomorphism and $T^{-1}S=ST^{-1}$ is transitive. If there exists a homeomorphism $A:X\to X,$ strongly transitive and that commutes with $T$ and $S.$ Then $T\times S$ is $\Delta$-transitive.  
\end{theorem}

\begin{rem} Observe the following.
\begin{enumerate}
\item The theorem is still true if $T$ is not necessarily an homeomorphism, but $S=T^{k}\tilde{S}$ for some $k\geq 1,$ $\tilde{S}:X\to X$ is a continuous map that commutes with $T,$ and $T^{k-1}\tilde{S}$ is transitive.
\item The theorem is still true if $A$ is not necessarily an homeomorphism, but $A=T.$
\end{enumerate}
\end{rem}

The main ingredients in the proof are the following theorem (similar to \cite{Moothathu_diagonal}, Prop. 1) and lemma (similar to \cite{KwOp_mixing}, Lem. 3).

\begin{theorem}\label{teo2}
Let $T_i: X \to X$ be continuous maps for $1 \leq i \leq H,$ where $H\geq 2.$ The following statements are equivalent.
\begin{enumerate}
	\item Given any non-empty sets $U_0, U_1, U_2, \ldots, U_H$, there exists $n$ for which
	\[ U_0 \cap T_1^{-n}U_1  \cap T_2^{-n}U_2 \cap \cdots \cap T_H^{-n}U_H \neq \emptyset.  \]
	\item $T_1 \times T_2 \times \cdots T_H$ is $\Delta$-transitive, i.e. there exists a $G_\delta$ set $Y \subset X$ such that for any $y \in Y$, we have $\{(T_1^ny, T_2^ny, \ldots, T_H^ny) : n \in \mathbb{N} \}$ is dense in $X^{H}$.
\end{enumerate}
In addition, if one of the transformation commutes with every other transformation, then the following statement is also equivalent to the ones above:
\begin{enumerate}
	\item[(3)] There exists $x \in X$ for which the set $\{(T_1^nx, T_2^nx, \ldots, T_H^nx) : n \in \mathbb{N} \}$ is dense in $X^{H}$.
\end{enumerate}
\end{theorem}

\begin{proof}
	Assume that (1) is satisfied. Let $\{B_k:k\in\NN\}$ be a countable base of open balls of $X.$ If
	$$
	Y=\bigcap_{(k_1,\ldots,k_H)\in\NN^H}\bigcup_{n\in \N}\bigcap_{i=1}^H T_i^{-n}(B_{k_i}).
	$$
	Then, by the Baire Category Theorem, $Y$ is a dense $G_{\delta}$ subset of $X,$ and by construction, every $x\in Y$ satisfies (2).
	
	Now assume that (2) is satisfied. Then, for every $x\in Y\cap U_0$ there exists $n\in \N$ such that $(T_1^nx, T_2^nx, \ldots , T_H^nx)\in U_1\times U_2\times \cdots \times U_H,$ then $\cap_{i=0}^H T_i^{-n}(U_i)\neq \emptyset.$
	
	To show that (3) implies (1), suppose that there exists $i$ such that $T_i$ commutes with all $T_j$ for $1\leq j\leq H$ and $A=\NN.$ Choose $k\in\NN$ such that $y=T_i^{k}(x)\in U \cap Y,$ then \[ \{(T_1^ny, T_2^ny, \ldots, T_H^ny): n \in \NN \} \] is dense in $X^{H},$ because $T_i$ commutes with all $T_j$ for $1\leq j\leq H.$ In particular, there exists $n\in\NN$ such that $(T_1^ny, T_2^ny, \cdots, T_H^ny)\in U_1\times U_2\times \cdots \times U_H,$ then $\cap_{i=0}^H T_i^{-n}(U_i)\neq \emptyset.$
\end{proof}

\begin{rem} It is clear that (2) implies (1). Under no commutative assumptions (1) implies (2) and (2) implies (3).
\end{rem}

\begin{lemma}\label{lemma1}
Suppose that $T:X\to X$ and $S:X\to X$  are continuous, $A: X\to X$ is a homeomorphism and that $T\times S:X^2\to X^2$ is topologically transitive. If $V_T$ and $V_S$ are nonempty open subsets of $X.$ Then there exists a sequence of integers $\{k_n\}_{n=0}^{\infty}$ such that $k_n>n$ for $n\geq 0$ and there exist sequences of nonempty open sets $\{V_T^{(n)}\}_{n=0}^{\infty}$ and $\{V_S^{(n)}\}_{n=0}^{\infty}$ such that 
for every $n\geq 0$ 
$$V_T^{(n)}\subset V_T, V_S^{(n)}\subset V_S,$$
$$
\begin{aligned}
T^{k_j} A^{-j} (V_T^{(n)})&\subset V_T \mbox{ for every }j=0,\ldots,n, \mbox{ and }\\
S^{k_j} A^{-j}(V_S^{(n)})&\subset V_S \mbox{ for every }j=0,\ldots,n.
\end{aligned}
$$ 
\end{lemma}

\begin{proof}
Let $V_T$ and $V_S$ be nonempty open subsets of $X$ and let $W:=V_T\times V_S.$ The proof follows by induction on $n.$ Base case ($n=0$): By transitivity of $T\times S$ there exists $k_0\in \N$ such that $(T\times S)^{k_0}W\cap W\neq\emptyset,$ i.e., $V_T^{(0)}:=V_T\cap T^{-k_0} V_T\subset V_T$ and $V_S^{(0)}:=V_S\cap T^{-k_0} V_S\subset V_S$ are nonempty open sets with the properties needed. Indeed, 
$$
\emptyset \neq T^{k_0}(V_T^{(0)})\subset T^{k_0}(V_T) \cap T^{k_0}(T^{-k_0} V_T)\subset T^{k_0}(V_T^{(0)})\subset T^{k_0}(V_T) \cap  V_T\subset V_T
$$
and similarly for $V_S^{(0)}.$ Inductive step: Assume there exists $\{k_0,\ldots,k_{n-1}\}$ and $\{V_T^{i}\}_{i=0}^{n-1},$ $\{V_S^{i}\}_{i=0}^{n-1}$ with the properties of the lemma. We define $U_T:=A^{-n}(V_T^{(n-1)}),$ $U_S:=A^{-n}(V_S^{(n-1)})$ and $U=U_T\times U_S.$ By transitivity of $T\times S$ there exists $k_n>n$ such that $(T\times S)^{k_n} U \cap W \neq \emptyset,$ i.e., $\tilde{U}_T:=U_T\cap T^{-k_{n}} V_T$ and $\tilde{U}_S:=U_S\cap T^{-k_{n}} V_S$ are nonempty open sets. We define $V_T^{(n)}:=V_T^{(n-1)}\cap A^n T^{-k_n}  (V_T)$ and  $V_S^{(n)}:=V_S^{(n-1)}\cap A^n S^{-k_n} (V_T).$ Observe that
$$
\emptyset\neq A^{n}(\tilde{U}_T)\subset A^n (U_T)\cap A^{n}(T^{-k_{n}} V_T)\subset V_T^{(n-1)}\cap A^{n}T^{-k_{n}} (V_T)=V_T^{(n)}
$$
and similarly 
$$
\emptyset\neq A^{n}(\tilde{U}_S)\subset A^n (U_S)\cap A^{n}(S^{-k_{n}} V_S)\subset V_S^{(n-1)}\cap A^{n}S^{-k_{n}} (V_S)=V_S^{(n)}.
$$
Also,
$$
\begin{aligned}
T^{k_n}A^{-n}(V_T^{(n)})&\subset T^{k_n}A^{-n}(V_T^{(n-1)}) \cap T^{k_n}A^{-n}(A^n T^{-k_n}  (V_T))\\
&= T^{k_n}A^{-n}(V_T^{(n-1)}) \cap T^{k_n} T^{-k_n} (V_T)\\
&\subset T^{k_n}A^{-n}(V_T^{(n-1)}) \cap V_T\\
&\subset V_T,
\end{aligned}
$$
similarly for $V_S^{(n)}$ (replacing $T$ by $S$).
Finally, by definition of $V_T^{(n)}$ we have that $V_T^{(n)}\subset V_T^{(n-1)},$ then for $j=0,\ldots, n-1$ 
$$
T^{k_j}A^{-j}(V_T^{(n)})\subset T^{k_j}A^{-j}(V_T^{(n-1)}),
$$
using the inductive hypothesis we obtain that 
$$
T^{k_j}A^{-j}(V_T^{(n)})\subset V_T,
$$
similarly for $V_S^{(n)}$ (replacing $T$ by $S$). This concludes the proof.
\end{proof}

Now we proceed to the prove of Theorem \ref{teo1}.

\begin{proof}[Proof of Theorem \ref{teo1}]
We use the Theorem \ref{teo2} to charactetise $\Delta$-transitivity. Let $U,V_T,U_S$ nonempty open subsets of $X.$ We use the strong transitivity of $A$ to find $M$ such that $\cup_{i=0}^M A^i U=X.$ We use proposition 4 in \cite{Moothathu_diagonal} to obtain that $T\times S$ is transitive. We use Lemma \ref{lemma1} to find  $\{k_j\}_{j=0}^M,$  $V_T^{(M)}$ and $V_S^{(M)}.$ We use the transitivity of $T^{-1}S$ to obtain $n\in \N$ such that $(T^{-1}S)^n(V_T^{(M)})\cap V_S^{(M)}\neq \emptyset.$ We fix an $a\in V_T^{(M)}$ and $b\in V_S^{(M)}$ such that $(T^{-1}S)^n a=b.$ We find a $y\in X$ such that $T^{n}y=a.$ We use strong transitivity of $A$ to find an $l\in \{0,\ldots, M\}$ and $x\in U$ such that $A^l x=y.$ We use Lemma \ref{lemma1} to prove that $T^{k_l}A^{-l}a\in V_T$ and $S^{k_l}A^{-l}b\in V_S.$ By construction, we have $x\in U$ such that
$$T^l A^{-l}  T^n A^l x\in V_T$$
and 
$$T^l A^{-l} (T^{-1}S)^n  T^n A^l x\in V_S.$$
By the hypotheses that $A$ and $T$ are homeomorphism, $A$ commutes with $T$ and $S,$ and that $ST^{-1}=T^{-1}S,$ we conclude the result.
\end{proof}

We find other sufficient condition for $\Delta$-transitivity motivated by the following property that appears without name in \cite{Moothathu_diagonal}.

\begin{definition}
We say that a dynamical systems $(X,T)$ satisfies the density condition if for any non empty open subsets $U, V \subset X$, there exists a nonempty open set $W \subset V$ and $n_0 \in \NN$ such that $W \subset T^n(U)$ for any $n \geq n_0$.
\end{definition}

\begin{proposition}\label{prop3}
Suppose that $(X,T)$ is a homeomorphism that satisfies the density condition and $T^{-1}S=ST^{-1}$ is syndetically transitive. Then $T\times S$ is $\Delta$-transitive.
\end{proposition}

\begin{rem} The proposition is still true if $T$ is not necessarily an homeomorphism, but $S=T^{k}\tilde{S}$ for some $k\geq 1,$ $\tilde{S}:X\to X$ is a continuous map that commutes with $T,$ and $T^{k-1}\tilde{S}$ is syndetically transitive.
\end{rem}

\begin{proof}
We use the Theorem \ref{teo2} to charactetise $\Delta$ - transitivity. Let $U,V_T,U_S$ nonempty open subsets of $X.$ We use the density condition to find $M$ and $W\subset V_T $ such that $W\subset T^n(U)$ for $n\geq M.$ We use the syndetic transitivity of $T^{-1}S$ to obtain $n\geq M$ such that $(T^{-1}S)^n(W)\cap V_S\neq \emptyset.$ We fix an $a\in W$ and $b\in V_S$ such that $(T^{-1}S)^n a=b.$ We find $x\in U$ such that $T^{n}x=a.$ By construction, we have $x\in U$ such that
$$T^n x\in V_T$$
and 
$$T^n (T^{-1}S)^n x\in V_S.$$
By the hypotheses that $ST^{-1}=T^{-1}S,$ we conclude the result.
\end{proof}

\subsection{Examples}

The following examples are easy to prove.

\begin{example}
Let $(\Sigma,\sigma)$ be a mixing subshift of finite type and $(p,q)$ a pair of different positive integers. If $T=\sigma^{p}$ and $S=\sigma^{q}.$ Then $T\times S$ is $\Delta$-transitive.
\end{example}

\begin{example}
Let $X=[0,1],$ $f:X\to X$ be a continuous and mixing map, and $(p,q)$ a pair of different positive integers. If $T=f^{p}$ and $S=f^{q}.$ Then $T\times S$ is $\Delta$-transitive.
\end{example}

\begin{example}
Let  $X=\{1,2,\ldots,N\}^{\N},$ $(X,F)$ be a topologically mixing cellular automata with anticipation $a$ and $\sigma:X\to X$ be the shift. If $b>a,$ $T=\sigma^b$ and $S=F.$ Then $T\times S$ is $\triangle$-transitive.
\end{example}

\section{ $\Delta$-transitivity II }\label{sec2}

Let $X$ be a compact metric space and $T_1,T_2,\ldots,T_H:X\to X$ be continuous maps for $H\geq 2.$ In this section we characterise topologically $\Delta$-transitivity for the product transformation $$T_1\times T_2\times \cdots \times T_H:X^{H}\to X^{H}.$$
The main theorem of this section is the following.

\begin{theorem}\label{teo4}
Suppose that $T_1,T_2,\ldots,T_H$ are weakly mixing and syndetically transitive, $T_1,\ldots, T_{H-1}$ are homeomorphism, $T_{i}^{-1}T_{i+1}=T_{i+1}T_{i}^{-1}$ satisfies the density condition for $i=1,\ldots,H-2$ and $T_{H-1}^{-1}T_{H}=T_{H}T_{H-1}^{-1}$ is transitive. If there exists a homeomorphism $A:X\to X,$ strongly transitive and that commutes with $T_i$ for every $i=1,2,\ldots,H.$ Then $T_1\times T_2\times \cdots \times T_H$ is $\Delta$-transitive. 
\end{theorem}

Immediately we obtain the following corollary.

\begin{corollary}\label{corteo4}
Suppose that $T_1,T_2,\ldots,T_H$ are minimal homeomorphisms, weakly mixing and commuting such that $T_{i+1}T_{i}^{-1}$ satisfies the density condition for $i=1,\ldots,H-2$ and $T_{H-1}^{-1}T_{H}=T_{H}T_{H-1}^{-1}$ is transitive. Then $T_1\times T_2\times \cdots \times T_H$ is $\Delta$-transitive.
\end{corollary}

\begin{proof}[Proof of Theorem \ref{teo4}]
We use Theorem \ref{teo2} to characterise $\Delta$-transitivity. Let $U,V_{T_1},$ $V_{T_2},$ $\ldots,$ $V_{T_H}$ be nonempty open subsets of $X.$ We use the strong transitivity of $A$ to find $M$ such that $\cup_{i=0}^M A^i U=X.$ We use a direct generalization of proposition 4 in \cite{Moothathu_diagonal} to obtain that $T_1\times T_2\times \cdots \times T_H$ is transitive. We use a direct generalisation of Lemma \ref{lemma1} to find $\{k_j\}_{j=0}^M,$  $V_{T_1}^{(M)}, \ldots, V_{T_H}^{(M)}$ with the properties that 
$$
V_{T_i}^{(M)}\subset V_{T_i}
$$
and
$$
T_i^{k_j}A^{-j}V_{T_i}^{(M)}\subset V_{T_i} \mbox{ for every }j=0,\ldots,M,
$$
for every $i=1,\ldots,H.$

We use the density condition of $T_{1}^{-1}T_2$ to find $W_2\subset V_{T_2}^{(M)}$ such that 
$W_2\subset T_1^{m} (V_{T_1}^{(M)})$ for every $m\geq M_1.$ Similarly, we find $W_3\subset V_{T_3}^{(M)}$ such that 
$W_3\subset T_2^{m} (W_2)$ for every $m\geq M_2.$ We repeat the argument up to finding $W_{H-1}\subset V_{T_{H-1}}^{(M)}$ such that 
$W_{H-1}\subset T_{H-2}^{m} (W_{H-2})$ for every $m\geq M_{H-2}.$

We use that $T_{H-1}^{-1}T_{H}$ is transitive to find $$n\geq \max\{M_i: i=1,\ldots,H-2\}$$ such that $$(T_{H-1}^{-1}T_{H})^n W_{H-1}\cap V_{T_H}^{(M)}\neq \emptyset.$$

We find $a_H\in V_{T_H}^{(M)},$ $a_{H-1}\in W_{H-1},$ $\ldots,$  $a_2\in W_{2}$, $a_1\in V_{T_{1}}^{(M)}$ such that 
$$
(T_{1}^{-1} T_2)^n a_1= a_2, (T_{2}^{-1} T_3)^n a_2= a_3,\ldots, (T_{H-1}^{-1} T_{H})^n a_{H-1}= a_{H}.
$$

We find a $y\in X$ such that $T^{n}y=a_1.$ We use strong transitivity of $A$ to find an $l\in \{0,\ldots, M\}$ and $x\in U$ such that $A^l x=y.$ We use a direct generalization of Lemma \ref{lemma1} to prove that $T_{i}^{k_l}A^{-l}a_i\in V_{T_i}$ for every $i=1,\ldots,H.$ By construction and the assumptions on $A$ we conclude the result.
\end{proof}

\begin{rem} Observe the following.
\begin{enumerate}
\item The theorem is still true if $T_i$ are not necessarily homeomorphisms, but
$$
\begin{aligned}
T_1&=T\\
T_2&=S_1T_1^{a_1}\\
T_3&=S_2T_2^{a_2}\\
&\vdots\\
T_H&=S_{H-1}T_{H-1}^{a_{H-1}},
\end{aligned}
$$
where $a_i>1,$ $S_i S_j=S_j S_i$ and $S_i T= T S_i$ for every $i,j=1,\ldots,H.$
\item The theorem is still true if $A$ is not necessarily an homeomorphism, but $A=T$ and 
$$
\begin{aligned}
T_1&=S_1T^{a_1}\\
T_2&=S_2T^{a_2}\\
&\vdots\\
T_H&=S_{H}T^{a_{H}},
\end{aligned}
$$
where $a_i\geq 1,$ $S_i S_j=S_j S_i$ and $S_i T= T S_i$ for every $i,j=1,\ldots,H.$
\end{enumerate}
\end{rem}

We can state a lemma analog to Proposition 3 (iv) in \cite{Moothathu_diagonal} that gives an example of a product transformation that is not $\Delta$ - transitive. 

\begin{lemma}\label{non-t-mixing-lemma}
Let $t\geq 2,$ $(p_1,\ldots,p_t)\in \N^t$ with $p_i\neq p_j$ for $i\neq j$, and $T_i= T^{p_i}$ for $i=1,\ldots,t.$ There exists a shift space $X$ such that for $T$ the shift action, $T_1\times T_2\times\cdots\times T_H$ is not $\triangle$-transitive.
\end{lemma}

\begin{proof}
We can assume that $p_1<\cdots<p_t.$ Let consider $X$ to be the subset of $\{0,1\}^{\N_0}$ such that for any $x\in X,$ $11$ does not appear in $x$ and if 
$$
1u_1 1 u_21\cdots 1 u_t 1
$$
is a subwords of $x$ for $u_1,\ldots,u_t\in \{0,1\}^{*}:=\cup_{n\in\N}\{0,1\}^n$ with $|u_i|:= \mbox{length}(u_i) =n_i$ for $i\in [t],$ we have that for $$n_0:=\min\left\{n\in\N: \min_{i\in [t)}\{n(p_{i+1}-p_{i})\}-1\geq 1\right\}$$ the vector 
$$
(n_1,\ldots,n_t)\notin \cup_{n\geq n_0}\{n(p_1,p_2-p_1,\ldots, p_t-p_{t-1})-(1,\ldots,1)\}.
$$
Clearly $X$ is $T$ invariant for $T:X\to X$ the action of the shift (that is continuous). By contradiction, assume that $T_1\times\cdots\times T_H$ is $\triangle$-transitive. Then there exists $n\in\N$ such that $ C(1) \cap_{i\in [t]} T^{-p_i n}C(1) \neq \emptyset,$ where $C(1):=\{w\in X: w_0=1\}.$ Therefore any sequence $x\in C(1) \cap_{i\in [t]} T^{-p_i n}C(1)$ is of the form $x=1u_1 1u_21\cdots 1 u_t 1\ldots ,$ where $u_1=0^{n p_1-1}$ and $u_{i+1}=0^{n (p_{i+1}- p_i)}$ for every $1\leq i <t,$ and by definition of $x\notin X,$ which is a contradiction.
\end{proof}

\section{Continuous coboundaries for the product of smooth functions}\label{sec3}

Let $(X, T)$ be a topological dynamical system, where $X$ is a compact metric space and $T: X \to X$ is a continuous map. We say a real-valued continuous function $f \in \mathcal{C}(X):= \mathcal{C}(X, \RR)$ is a \textit{coboundary} if there exists $g \in \mathcal{C}(X)$ such that $f = g - g\circ T.$ There are two well studied topological conditions under which a continuous function $f$ is a coboundary: the first are provided by Liv\v{s}ic theorem for transitive, surjective continuous maps and the second by Gottschalk-Hedlund theorem for minimal systems. We use here only Liv\v{s}ic theorem. 

\begin{theorem}[Liv\v{s}ic, \cite{Livsic_homology, Livsic_chomology} ]\label{original_Livsic_theorem}
	Let $T:X\to X$ be a transitive, surjective continuous map that satisfies the closing property (see Definition \ref{CP}). Let $f:X\to \RR$ be a $\alpha$-H\"older map (see Definition $\ref{def:alpha-Holder}$) such that for any $p$ in $X$ with $T^k p=p$ for some $k$ in $\NN,$ the sum  $\sum_{j=0}^{k-1}f(T^j p)$ is equal to zero. Then there exists $g:X\to \RR$ such that it is $\alpha$-H\"older and $g - g\circ T =f.$
\end{theorem}

In this section we apply the results on $\Delta$ - transitivity of 
$$\mathcal T:=T_1\times T_2 \times \cdots \times T_H$$ to provide a version of Liv\v{s}ic that only requires to check a condition on the diagonal subspace $\Delta := \{(x, x, \ldots, x) \in X^{H}: x \in X \}$ of $\mathcal X:= X^{H}.$

\subsection{Liv\v{s}ic Theorem for $\Delta$ - transitive products}\label{Main_theorem}

Let $(X,d)$ be a compact metric space, that we assume to be separable, complete, and without isolated points, and let $T:X\to X$ be a homeomorphism. We recall the closing property. 
\begin{definition}\label{CP} We say that a dynamical systems $(X,T)$ satisfies the \textit{closing property} if there exists $D,\delta,\delta_0>0$ such that for all $x$ in $X$ and $k$ in $\NN$ with $d(x,T^k x)<\delta_0$ there exists $p$ in $X$ such that $T^k p=p$ and such that $$d(T^i x, T^i p)\leq D d(x,T^k x)e^{-\delta \min\{i,k-i\}} \mbox{ for all } 0\leq i \leq k.$$
\end{definition}
This conditions (and the density condition) are satisfied by many well studied dynamical systems, for example, subshifts of finite type and hyperbolic diffeomorphisms on compact manifolds.\\

We define the space of $\alpha$-H\"older functions on $X.$

\begin{definition}[$\alpha$-H\"older]\label{def:alpha-Holder}
 We say that a function $f:X\to \RR$ is \textit{$\alpha$-H\"older} for $\alpha\in(0,1]$ if there exists a constant $C>0$ such that for every pair of points $x,y$ in $X,$
 $$
 |f(x)-f(y)|\leq C d (x,y)^{\alpha}.
 $$
\end{definition}

Given $H$ in $\NN$ and the continuous maps $T_i:X\to X$ for $1\leq i \leq H,$ we obtain a Liv\v{s}ic Theorem for multiple ergodic averages, that is, we consider the metric space $(\mathcal X:=X^{H},d_H)$ with $$d_H(x,y):=\max_{1\leq i  \leq H}d(x_i,y_i),$$ and the continuous map $$\mathcal T:=T_1\times T_2 \times \cdots \times T_H : \mathcal X \to \mathcal X.$$

Under the assumption that $\mathcal T$ is $\Delta$ - transitive, one can use Theorem \ref{original_Livsic_theorem} to obtain the following nonconventional Liv\v{s}ic theorem.

\begin{theorem}[Nonconventional Liv\v{s}ic]\label{Nonconventional_Livsic_theorem} Let $f:\mathcal X\to \RR$ be an $\alpha$-H\"older map. If $\mathcal T$ is $\Delta$ - transitive and satisfies the closing property.
Then the following statements are equivalent.
	\begin{enumerate}[(i)]
		\item \label{thm_one}There exists a non empty open subset $U\subset X$ such that $$\sup_{n} \sup_{x\in U} \left| \sum_{j=0}^{n} f({\mathcal T}^{j} \hat{x}) \right| < \infty \text{ where } \hat{x} := (x, x, \ldots, x)\in\Delta.$$ 		
		\item  \label{thm_two} For any $p \in \mathcal X$ such that $T^k p = p$, we have
		\[\sum_{j=0}^{k-1} f({\mathcal T}^j p) = 0. \]
		\item  \label{thm_three} There exists $V:\mathcal X\to \RR$ such that it is $\alpha$-H\"older and $V - V \circ {\mathcal T}=f.$
		\item  \label{thm_fourth} There exists $D > 0$ such that
		\[\sup_{n} \sup_{x \in \mathcal X} \left| \sum_{j=0}^{n} f({\mathcal T}^j x)  \right| < D. \]
	\end{enumerate}
\end{theorem}

\begin{rem}
We note that statement (i) only requires one to check the boundedness of the sums along $\Delta$, as opposed to the entire space $\mathcal X,$ that would be the case of applying directly Liv\v{s}ic Theorem to the map $\mathcal T : \mathcal X \to \mathcal X. $ This may be an easier property to check than the others, and leaves a possibility for a simulation to find if such open set $U$ exists.\\
\end{rem}

We notice that the theorem can be written in terms of $(X,T_1,\ldots,T_H)$ instead of $(\mathcal X,\mathcal T).$ Indeed, Theorem \ref{teo4} gives sufficient conditions on $T_1,\ldots,T_H$ such that $\mathcal T$ is $\Delta$ - transitive It is also possible to prove that  $\mathcal T$ satisfies the closing property if each $T_i$ satisfies it for every $i=1,\ldots, H.$ Indeed, we have the following lemma.

\begin{lemma}\label{lemma_product_CP}
If $(X, T_i)$ satisfies the closing property for every $1\leq i \leq H.$ Then $({\mathcal X}, {\mathcal T}=T_1\times \cdots\times T_H)$ also satisfies the closing property.
\end{lemma}

\begin{proof}
Assume that $(X, T_i)$ satisfies the closing property for every $1\leq i \leq H.$ By the definition, for every $1\leq i \leq H,$ there exists $D_i,\delta_i,\delta_0(i)>0$ such that for all $x$ in $X$ and $k$ in $\NN$ with $d(x,T_i^k x)<\delta_0(i)$ there exists $p_i$ in $X$ such that $T_i^k p_i=p_i$ and such that $d(T_i^j x_i, T_i^j p_i)\leq D_i d(x_i,T_i^k x_i)e^{-\delta_i \min\{j,k-j\}}$  for all  $j=0,\ldots,k.$ Therefore, choosing $D=\max_{1\leq i \leq H}D_i,\delta=\min_{1\leq i \leq H}\delta_i $ and $ \delta_0=\min_{1\leq i \leq H}\delta_0(i),$ we obtain that for all $x=(x_1,\ldots,x_H)$ in ${\mathcal X}$ and $k$ in $\NN$ with $d_H(x,{\mathcal T}^k x)<\delta_0$ there exists $p=(p_1,\ldots,p_H)$ in ${\mathcal X}$ such that ${\mathcal T}^k p=p$ and such that $d_H({\mathcal T}^j x, {\mathcal T}^j p)\leq D d_H(x,{\mathcal T}^k x)e^{-\delta \min\{j,k-j\}}$  for all  $j=0,\ldots,k.$
\end{proof}

\subsection{Topological analog of a result by Assani}

Let consider $(\mathcal X, {\mathcal T}),$ where $\mathcal X= X^{H}$ and ${\mathcal T} = T_1 \times T_2 \times \cdots \times T_H,$ for some $H\geq 2.$ Assume that each $T_i:X\to X$ is continuous for any $i=1,\ldots,H.$

We say that $F = \otimes_{i=1}^H f_i$ defined by $\otimes_{i=1}^H f_i (x_1,\ldots, x_H)=\prod_{i=1}^H f_i (x_i)$ is a measure-theoretic coboundary if there exists a real-valued function $V$ on $\mathcal X$ such that the map $x \mapsto V \circ {\mathcal T} (x, x, \ldots, x)$ is essentially bounded and $\mu$-measurable, and $$F(x, x, \ldots, x) = V(x, x, \ldots, x) - V \circ {\mathcal T}(x, x, \ldots, x)$$ for $\mu$-a.e $x \in X$.

A natural problem is finding conditions on $(X,\mu,T_1,\ldots,T_H)$ so that $F$ is a measure-theoretic coboundary. Recently, I. Assani \cite{Assani_coboundary} obtained a necessary and sufficient condition for this case, see Theorem \ref{AssaniThm}.

\begin{rem}
	The proof uses the \textit{diagonal-orbit measure} of $\mu$ (cf. \cite[Definition 1.2]{Assani_coboundary}), which is a tool introduced in \cite{Assani_v1} to study the pointwise convergence of nonconventional ergodic averages. Diagonal-orbit measures were used to describe the behavior of the nonconventional ergodic sums and averages along the orbit of the diagonal space $\Delta \subset \mathcal X$ iterated by the map ${\mathcal T}.$
\end{rem}

	By contrast to Assani's work, we show that for a certain class of topological system $(X, T_1, T_2, \ldots, T_H)$ and under certain smoothness assumption of the functions $f_1, f_2, \ldots, f_H,$ one can find necessary and sufficient condition so that $\otimes_{i=1}^H f_i$ is a topologically smooth coboundary with respect to the map ${\mathcal T}$. This answers a question that was raised during a discussion between the second author and S. Donoso: When is the product function $\bigotimes_{i=1}^H f_i $ a continuous coboundary?
	
\begin{corollary}\label{Nonconventional_Livsic_corollary} Let $f_1,\ldots,f_H$ be $\alpha$-H\"older maps. Suppose that $(X,T_1,T_2,\ldots,T_H)$ satisfies the hypotheses of Theorem \ref{teo4}. Then the following statements are equivalent.
	\begin{enumerate}
		\item \label{N_L_corollary_1} There exists a non-empty open subset $U \subset X$ such that		
		$$\sup_{N} \sup_{x\in U}\left| \sum_{j=0}^{N} \prod_{i=1}^H f_i (T_i^{j} x) \right| < \infty.$$
		\item \label{N_L_corollary_2} The product of the function is a $\alpha$-H\"older coboundary, i.e. if ${\mathcal T}=T_1\times T_2 \times \cdots \times T_H,$ there exists $V$ $\alpha$-H\"older such that
		$$
		\bigotimes_{i=1}^H f_i = V - V\circ {\mathcal T}.
		$$
	\end{enumerate}
\end{corollary}

The proof is direct from Theorem \ref{teo4} and \ref{Nonconventional_Livsic_corollary}, and the following lemma.

\begin{lemma}\label{alpha-holder-product}
	If $f_1,\ldots,f_H$ are $\alpha$-H\"older. Then $\bigotimes_{i=1}^H f_i $ is $\alpha$-H\"older.
\end{lemma}

In the proof we use the following notation. Given $H \geq 2,$ $x=(x_1,x_2,\ldots,x_H)$ and $f := \bigotimes_{i=1}^H f_i ,$ we define $f^* := \bigotimes_{i=1}^{H-1} f_i $ and $x^* := (x_1,x_2,\ldots,x_{H-1}).$

\begin{proof}[Proof of Corollary \ref{Nonconventional_Livsic_corollary}.]
	We proceed by an induction on $H$. The base case $H=1$ is trivial. Assume that $f_1,\ldots,f_H$ are $\alpha$-H\"older and the property is valid for $H-1$ (i.e. $f^*$ is $\alpha$-H\"older). Let $x,y\in {\mathcal X}.$  We have that
	\begin{align*}
	|f(x)-f(y)|
	&= |f^*(x^*)f_H(x_H) - f^*(y^*)f_H(y_H)| \\
	&\leq |f^*(x^*)||f_H(x_H) - f_H(y_H)| + |f_H(y_H)||f^*(x^*) - f^*(y^*)|.
	\end{align*}
	Using that $f_H$ is  $\alpha$-H\"older we obtain that $|f_H(x_H) - f_H(y_H)|\leq C_1 d(x_H,y_H)^{\alpha}$ for some $C_1>0.$ Using that $f^*$ is $\alpha$-H\"older we obtain that $|f^*(x^*) - f^*(y^*)|\leq C_2 d_{m-1}(x^*,y^*)^{\alpha}$ for some $C_2>0.$ Using the compactness of the space $X$ (therefore also of $X^{H-1}$) and the continuity of $f_{H}$ and $f^*,$ we have that there exist constants $C_3,C_4>0$ such that $C_3:=\sup\{ |f_{H}(x)|:x\in X\}$ and $C_4:=\sup\{ |f^*(x)|:x\in X^{H-1}\}.$ Therefore, for $C:=\max\{C_4 C_1,C_3 C_2\},$ we have that
	$$
	|f(x)-f(y)|\leq C_4 C_1 d(x_H,y_H)^{\alpha}+ C_3 C_2 d_{H-1}(x^*,y^*)^{\alpha}\leq 2 C d_{H}(x,y)^{\alpha}.
	$$
	
\end{proof}


\section*{Acknowledgments}
  We are grateful for Mario Ponce, Sebasti\'an Donoso, and Godofredo Iommi for having fruitful discussions with us. We also thank Idris Assani for his comments.

\end{document}